\documentclass[letterpaper, 10 pt, conference]{ieeeconf}  

\IEEEoverridecommandlockouts                              

\usepackage{amsmath} 
\usepackage{amssymb}  
\newtheorem{lemma}{Lemma}
\newtheorem{theorem}{Theorem}

\newtheorem{proposition}{Proposition}
\newtheorem{remark}{Remark}
\newtheorem{corollary}{Corollary}
\usepackage{algorithm}
\usepackage{algpseudocode}

\let\labelindent\relax
\usepackage{enumitem}

\title{\LARGE \bf
Polynomial-Time Approximation for Nonconvex Optimization Problems with an L1-Constraint
}

\author{Yonatan Mintz and Anil Aswani
\thanks{*This work was supported in part by NSF Award CMMI-1450963 and the Philippine-California Advanced Research Institutes (PCARI).}
\thanks{Yonatan Mintz and Anil Aswani are with the Department of Industrial Engineering and Operations Research, University of California, Berkeley, CA 94720, USA 
        {\tt\small ymintz@berkeley.edu, aaswani@berkeley.edu}}%
}

\begin{document}

\maketitle
\thispagestyle{empty}
\pagestyle{empty}

\begin{abstract}
Nonconvex optimization problems with an L1-constraint are ubiquitous, and are found in many application domains including: optimal control of hybrid systems, machine learning and statistics, and operations research.  This paper shows that nonconvex optimization problems with an L1-constraint can be approximately solved in polynomial time.  We first show that nonlinear integer programs with an L1-constraint can be solved in a number of oracle steps that is polynomial in the dimension of the decision variable, for each fixed radius of the L1-constraint.  When specialized to polynomial integer programs, our result shows that such problems have a time complexity that is polynomial in simultaneously both the dimension of the decision variables and number of constraints, for each fixed radius of the L1-constraint.  We prove this result using a geometric argument that leverages ideas from stochastic process theory and from the theory of convex bodies in high-dimensional spaces.  We conclude by providing an additive polynomial time approximation scheme (PTAS) for continuous optimization of Lipschitz functions subject to Lipschitz constraints intersected with an L1-constraint, and we sketch a generalization to mixed-integer optimization.
\end{abstract}

\section{Introduction}
Nonconvex optimization with an L1-constraint arises when solving many practical problems.  For instance, an L1-constraint in optimal control of switched systems \cite{xu2004optimal,sun2005analysis,aswani2009monotone,vasudevan2013consistent,vasudevan2013consistent2,wardi2015switched} can limit the total number of mode changes.  In machine learning and statistics, an L1-constraint provides sparsity-promoting regularization for nonlinear regression models \cite{ribbing2007lasso,buhlmann2011non} and neural networks in deep learning \cite{glorot2011deep,srivastava2014dropout}.  Operations research frequently uses L1-constraints to represent capacity or budget constraints \cite{raghavan1987randomized,raghavan1988probabilistic,srinivasan1999improved}.

Given the ubiquitousness of nonconvex optimization problems with L1-constraints, this paper uses the notion of fixed-parameter complexity \cite{downey1999,niedermeier2006} to study their computational complexity.  We show these problems can be approximately solved in polynomial time for a fixed radius of the L1-constraint.  For polynomial integer programming, we prove its computational complexity is polynomial in the number of constraints and dimension of the decision variables.  For continuous optimization, we construct an additive polynomial time approximation scheme (PTAS) for optimization problems with Lipschitz continuous objectives and constraints, and we generalize this result to mixed-integer optimization.

\subsection{Approximation of Integer Programs}

Though integer programming is \textsf{NP}-complete \cite{borosh1976,von1978,kannan1978,papadimitriou1981,pia2016}, fixed-parameter complexity \cite{downey1999,niedermeier2006} gives a finer classification.  The integer linear program $\min\{c'x\ |\ Ax \leq b, x\in\mathbb{Z}^n\}$ with $c\in\mathbb{Q}^n$, $b\in\mathbb{Z}^m$, and $A \in \mathbb{Z}^{m\times n}$ can be solved in polynomial time in (1) number of constraints $m$ when dimension $n$ is fixed, and (2) dimension $n$ when number of constraints $m$ is fixed \cite{papadimitriou1981,clarkson1995,eisenbrand2003,frank1987,kannan1987,lenstra1983}.  The problem $\min\{p(x)\ |\ Ax\leq b, x\in\mathbb{Z}^n\}$, where $p(x)$ is a polynomial, can be approximately solved \cite{de2006,hildebrand2016,del2016} in polynomial time for fixed dimension $n$.  

However, little is known about which other classes of integer programs can be solved in fixed-parameter polynomial time.  This paper proves that the nonlinear integer programs
\begin{equation}
\label{eqn:gip}
\min\{f(x)\ |\ g(x) \leq 0, \|x\|_1 \leq \lambda, x\in\mathbb{Z}^n\}
\end{equation}
can be solved in polynomial-time oracle complexity for fixed $\lambda\in\mathbb{R}_+$, where $\|\cdot\|_1$ is the usual $\ell_1$-norm.  If $f,g$ can be computed in polynomial time, then our result implies (\ref{eqn:gip}) can be solved in polynomial (in simultaneously both dimension $n$ and number of constraints) time for fixed $\lambda$.

\subsection{Approximation of Lipschitz Continuous Programs}

Additive polynomial time approximation schemes (PTAS) for continuous optimization of Lipschitz functions over the unit simplex are known \cite{deklerk2008,deklerk2006,deklerk2015}.  In this paper, we generalize these lengthy and unintuitive results by using our results on (\ref{eqn:gip}) to construct an additive PTAS for optimizing
\begin{equation}
\label{eqn:formulation}
\min \big\{f(x)\ \big|\ g(x) \leq 0, \|x\|_1 \leq \lambda, x\in\mathbb{R}^n\big\}, 
\end{equation}
where $f,g$ are Lipschitz: $|f(x)-f(y)|\leq\kappa\|x-y\|_\infty$ and $|g_i(x)-g_i(y)|\leq\kappa\|x-y\|_\infty$ for all $i$ and some $\kappa \in\mathbb{R}_+$.  Our results generalize to the mixed-integer optimization $\min \big\{f(x,y)\ \big|\ g(x,y) \leq 0, \|x\|_1 \leq \lambda, x\in\mathbb{Z}^n, y\in\mathbb{R}^m\big\}$ when the functions $f,g$ are convex in $y$ for each fixed $x$.

\subsection{Outline}

Section \ref{section:prelim} covers preliminaries, and Sect. \ref{section:alg} develops our algorithm to solve (\ref{eqn:gip}).  We use stochastic process theory to prove the number of integer-valued vectors in the scaled $\ell_1$-ball is upper-bounded by a polynomial in $n$, which shows our algorithm has polynomial complexity on problems with L1-constraints.  Next, we give an algorithm to compute an upper bound on the number of oracle operations required to solve a given integer program $\min\{f(x)\ |\ g(x) \leq 0, x\in\mathbb{Z}^n\}$ for when: $g$ is a convex function that can be optimized in polynomial time, and the continuous relaxation of the feasible set $\{x : g(x)\leq 0\}$ is bounded.  We conclude with Sect. \ref{section:ptas}, which gives an additive PTAS for optimizing (\ref{eqn:formulation}).

\section{Preliminaries}
\label{section:prelim}

Let $\mathbb{R}^n$, $\mathbb{Q}^n$, $\mathbb{Z}^n$, and $\mathbb{N}^n$ be the set of $n$-dimensional real-, rational-, integer-, and natural number-valued vectors, respectively, where the natural numbers are $\mathbb{N} = \{0,1,\ldots\}$.  Let $[r] := \{1,2,\ldots,r\}$, and let $\lfloor r\rfloor$ be the largest integer smaller than $r$.  The $\mathrm{round}(o)$ function rounds each component of the vector $o$ to a nearest integer, and $\mathbf{1}(A)$ is an indicator function that is 1 if $A$ is true, and 0 if $A$ is false.

Recall the usual inner product $\langle g,x\rangle = \sum_j g_jx_j$, and let $\circ$ be the Hadamard or elementwise product operator, such that  $z = x \circ y \iff z_i = x_i\cdot y_i, \; \forall i\in[n]$.  We use the $\ell_p$ norm notation: $\|x\|_1 = \sum_j |x_j|$ for the $\ell_1$-norm, $\|x\|_\infty = \max_j |x_j|$ for the $\ell_\infty$-norm, and $\|x\|_2 = \langle x,x\rangle^{1/2}$ for the $\ell_2$-norm.  Define $\mathcal{B}_1 = \{x : \|x\|_1 \leq 1\}$, $\mathcal{B}_2 = \{x : \|x\|_2 \leq 1\}$, and $\mathcal{B}_\infty = \{x : \|x\|_\infty\leq 1\}$ to be the unit $\ell_1$-, $\ell_2$-, and $\ell_\infty$-balls centered at the origin, respectively.  For any $\lambda\in\mathbb{R}_+$ and set $K\subset \mathbb{R}^n$, the scaled set is $\lambda K := \{\lambda x : x\in K\}$.  For instance, $\lambda\mathcal{B}_1 = \{x : \|x\|_1\leq\lambda\}$, $\lambda\mathcal{B}_2=\{x:\|x\|_2\leq\lambda\}$, and $\lambda\mathcal{B}_\infty = \{x:\|x\|_\infty\leq\lambda\}$.

If $K,D \subset \mathbb{R}^n$ are convex sets, then the covering number of $K$ by copies of a ball $D$ is defined as the quantity
\begin{equation}
N(K,D) = \min_V\textstyle \Big\{\# V : \bigcup_{v\in V} \big(D \oplus v\big)\supseteq K\Big\},
\end{equation}
where $\# V$ is the cardinality of the set $V$, and $\oplus$ denotes the Minkowski summation operator defined as $A \oplus B = \{a + b : a\in A, b\in B\}$.  A related value is the packing number of $K$ by copies of a ball $D$, which is defined as the quantity
\begin{multline}
P(K,D) = \max_V\textstyle \Big\{\# V : \bigcup_{v\in V} \big(D \oplus v\big)\subseteq K\ \wedge \\
\textstyle\big(D \oplus v_i\big) \cap \big(D \oplus v_j\big) = \emptyset\ \mathrm{for\ all}\ v_i,v_j\in V\ \mathrm{with}\ v_i\neq v_j\Big\}.
\end{multline}
A basic inequality \cite{boucheron2013} relating these two quantities is $P(K,D) \leq N(K,D) \leq P(K,D/2)$.

\section{Polynomial-Time Solvability with a Bounded L1-Constraint}
\label{section:alg}

We show (\ref{eqn:gip}) can be solved in polynomial time for fixed $\lambda$.  We use the theory of stochastic processes and convex bodies in high-dimensional spaces to characterize the complexity of scaled $\ell_p$-balls, which we then use to upper bound the number of integers within scaled $\ell_p$-balls centered at the origin.  Surprisingly, we find that the number of integers in the scaled $\ell_1$-ball is polynomial in dimension when the radius of the ball is fixed, which is in sharp contrast to the number of integers within the scaled $\ell_\infty$-ball.  We use this result to design an algorithm to solve (\ref{eqn:gip}), and we prove its polynomial-time complexity.  Our algorithm is then generalized to integer programs with a weighted L1-constraint.

\subsection{Characterizing the Integer Complexity of Lp-Balls}

Covering numbers provide one useful measure of the complexity of a set in Euclidean space.  Exactly determining covering numbers is difficult, but fortunately stochastic process theory provides several approaches for bounding the covering number.  One approach is Sudakov's minoration \cite{ledoux1991,sudakov1971,vershynin2009}, which relates the covering number of a symmetric convex set $K$ to its Gaussian mean width:

\begin{proposition}[Sudakov's Minoration \cite{sudakov1971}]
Let $K \subset \mathbb{R}^n$ be a symmetric convex set, and recall that $N(K,r\mathcal{B}_2)$ is the covering number of $K$ by $\ell_2$-balls with radius $r$.  We have the bound
\begin{equation}
\label{eqn:sudakov}
\sqrt{\log N(K,r\mathcal{B}_2)} \leq \frac{1}{2r}\mathbb{E}\Big(\max_{x\in K} \langle g,x\rangle\Big),
\end{equation}
where $g \in \mathbb{R}^n$ is a vector whose entries are iid Gaussian random variables with zero mean and unit variance.
\end{proposition}

Another useful approach is a basic volume-based inequality \cite{vershynin2009} for bounding the covering number; however, these bounds can be loose.

\begin{proposition}[Estimate of Covering Number \cite{vershynin2009}]
\label{prop:vecn}
For any symmetric convex sets $K,D\subset\mathbb{R}^n$ we have
\begin{equation}
\label{eqn:vbineq}
\frac{\mathrm{vol}(K)}{\mathrm{vol}(D)} \leq N(K, D) \leq \frac{\mathrm{vol}(K\oplus\frac{1}{2}D)}{\mathrm{vol}(D)}.
\end{equation}
When $K/\lambda = D/r$, this simplifies to $(\frac{\lambda}{r})^n \leq N(\lambda K, rK) \leq (2 + \frac{\lambda}{r})^n$.
\end{proposition}

The above approaches for bounding the covering number of a symmetric convex set can be used to derive bounds for the $\ell_1$-, $\ell_2$-, and $\ell_\infty$-balls.  Asymptotic bounds are found in \cite{vershynin2009}, but we need non-asymptotic bounds for our purposes.

\begin{proposition}
\label{prop:covnum}
Recall that $\lambda\mathcal{B}_1$, $\lambda\mathcal{B}_2$, and $\lambda\mathcal{B}_\infty$ are the $\ell_1$-, $\ell_2$-, and $\ell_\infty$-balls centered at the origin and with radius $\lambda$.  If $r \leq \lambda$, then we have
\begin{equation}
\begin{aligned}
2n &\leq& N(\lambda\mathcal{B}_1,r\mathcal{B}_\infty) &\leq& n^{(\lambda/\sqrt{2}r)^2}\\
2n &\leq &N(\lambda\mathcal{B}_2,r\mathcal{B}_\infty) &\leq& \textstyle(2 + \frac{\lambda}{r})^n\\
\textstyle(\frac{\lambda}{r})^n &\leq &N(\lambda\mathcal{B}_\infty,r\mathcal{B}_\infty) &\leq& \textstyle(2 + \frac{\lambda}{r})^n
\end{aligned}
\end{equation}
The upper bounds hold unconditionally, that is for all $\lambda \geq 0$.
\end{proposition}

\begin{proof}
Since $\lambda\mathcal{B}_1$ is symmetric and convex, we can upper bound its covering number using the Sudakov minoration.  The first step is to compute the Gaussian mean width, which is achieved by noting that H\"{o}lder's inequality and the symmetry of $\lambda\mathcal{B}_1$ give
\begin{equation}
\mathbb{E}\Big(\max_{x\in \lambda\mathcal{B}_1} \langle g,x\rangle\Big) \leq \mathbb{E}\Big(\max_{x\in\lambda\mathcal{B}_1} \|x\|_1 \cdot\max_j|g_j|\Big) \leq \lambda\sqrt{2\log n},
\end{equation}
where we have used the elementary bound $\mathbb{E}(\max_j|g_j|) \leq \sqrt{2\log n}$.  Combining this with Sudakov's minoration gives
\begin{equation}
\sqrt{\log N(\lambda\mathcal{B}_1, r\mathcal{B}_2)} \leq \lambda\sqrt{\log n}/\sqrt{2}r,
\end{equation}
which simplifies to $N(\lambda\mathcal{B}_1,r\mathcal{B}_2) \leq n^{(\lambda/\sqrt{2}r)^2}$.  The upper bound follows from $N(\lambda\mathcal{B}_1,r\mathcal{B}_\infty)\leq N(\lambda\mathcal{B}_1,r\mathcal{B}_2)$ since $\mathcal{B}_2 \subseteq \mathcal{B}_\infty$.  The lower bound is from \cite{vershynin2009}, and the argument is that one copy of $r\mathcal{B}_2$ is needed for each vertex of $\lambda\mathcal{B}_1$.

Bounds for the covering number of $\lambda\mathcal{B}_\infty$ follow from the simplified expression in Proposition \ref{prop:vecn} for the situation where $K/\lambda = D/r = \mathcal{B}_\infty$. 

An upper bound for the covering number of $\lambda\mathcal{B}_2$ follows by noting $N(\lambda\mathcal{B}_2,r\mathcal{B}_\infty) \leq N(\lambda\mathcal{B}_\infty,r\mathcal{B}_\infty)$ since $\mathcal{B}_2 \subseteq \mathcal{B}_\infty$, and a lower bound follows by noting $N(\lambda\mathcal{B}_2,r\mathcal{B}_\infty) \geq N(\lambda\mathcal{B}_1,r\mathcal{B}_\infty)$ since $\mathcal{B}_2 \supseteq \mathcal{B}_1$.
\end{proof}

\begin{remark}
This says the number of $\ell_\infty$-balls we need to cover an $\ell_1$-ball is polynomial in $n$ for fixed $\lambda/r$.  This is significant because in general (e.g., the $\ell_\infty$-balls) covering an $n$-dimensional convex body requires an exponential in $n$ number of balls.  
\end{remark}


The reason for our interest in covering numbers is that they can be used to count the number of integers within a convex set.  In particular, we have the bounds:
\begin{theorem}
\label{thm:latcount}
Let $K\subset\mathbb{R}^n$ be a convex set.  For any $\delta \in (0,2)$ we have that
\begin{equation}
\textstyle N(K,\frac{2}{2-\delta}\mathcal{B}_\infty)\leq\#(K\cap\mathbb{Z}^n) \leq N(K, \frac{1}{2+\delta}\mathcal{B}_\infty).
\end{equation}
\end{theorem}

\begin{proof}
The upper bound follows if we can show that two integer-valued vectors cannot lie within the same $\frac{1}{2+\delta}\mathcal{B}_\infty$ covering-ball.  To prove this, suppose the opposite is true.  Then there exists some point $o$ such that $u,v$ are two integer-valued vectors with $u\neq v$ and $u,v\in\{x : \|x-o\|_\infty\leq\frac{1}{2+\delta}\}$.  Using the triangle inequality gives $\|u-v\|_\infty \leq \|u-o\|_\infty + \|v-o\|_\infty\leq \frac{2}{2+\delta}<1$ which is a contradiction because $\|u-v\|_\infty \geq 1$.

The lower bound follows by using the basic inequality $N(K,D) \leq P(K,D/2)$ and noting that a packing ball $\frac{1}{2-\delta}\mathcal{B}_\infty$ must contain at least one integer, since if $o$ is the center of the packing ball then $\|o - \mathrm{round}(o)\|_\infty \leq \frac{1}{2} < \frac{1}{2-\delta}$ by definition; meaning $\mathrm{round}(o)$ is an integer-valued vector in the packing ball $\frac{1}{2-\delta}\mathcal{B}_\infty$ centered at $o$.
\end{proof}

\begin{remark}
A simplified set of bounds implied by the above theorem are
\begin{equation}
\textstyle N(K,2\mathcal{B}_\infty)\leq\#(K\cap\mathbb{Z}^n) \leq N(K, \frac{1}{4}\mathcal{B}_\infty).
\end{equation}
\end{remark}

The above result can be used to bound the number of integer-valued vectors within the $\ell_1$-, $\ell_2$-, and $\ell_\infty$-balls.

\begin{corollary}
\label{cor:intbound}
Recall that $\lambda\mathcal{B}_1$, $\lambda\mathcal{B}_2$, and $\lambda\mathcal{B}_\infty$ are the $\ell_1$-, $\ell_2$-, and $\ell_\infty$-balls centered at the origin and with radius $\lambda$.  If $\lambda\geq 1$, then for any $\delta\in(0,1)$ we have that
\begin{equation}
\begin{aligned}
2n &\leq& \#(\lambda\mathcal{B}_1\cap\mathbb{Z}^n) &\leq& n^{((2+\delta)\lfloor\lambda\rfloor)^2/2}\\
2n &\leq &\#(\lambda\mathcal{B}_2\cap\mathbb{Z}^n) &\leq& \textstyle(1+2\lfloor\lambda\rfloor)^n\\
\textstyle(1+2\lfloor\lambda\rfloor)^n &= &\#(\lambda\mathcal{B}_\infty\cap\mathbb{Z}^n) &&
\end{aligned}
\end{equation}
The upper bounds hold unconditionally, that is for all $\lambda \geq 0$.
\end{corollary}

\begin{proof}
If $x\in\mathbb{Z}^n$, then $\|x\|_1\leq\lambda$ implies $\|x\|_1\leq\lfloor\lambda\rfloor$, which means $(\lambda\mathcal{B}_1\cap\mathbb{Z}^n) = (\lfloor\lambda\rfloor\mathcal{B}_1\cap\mathbb{Z}^n)$.  The bounds for $\#(\lambda\mathcal{B}_1\cap\mathbb{Z}^n)$ then follow by Theorem \ref{thm:latcount} and Proposition \ref{prop:covnum}.  Similarly, if $x\in\mathbb{Z}^n$, then $\|x\|_\infty\leq\lambda$ implies $\|x\|_\infty\leq\lfloor\lambda\rfloor$, which means $(\lambda\mathcal{B}_\infty\cap\mathbb{Z}^n) = (\lfloor\lambda\rfloor\mathcal{B}_\infty\cap\mathbb{Z}^n)$.  However, we have $\#(\lfloor\lambda\rfloor\mathcal{B}_\infty\cap\mathbb{Z}^n) = (1+2\lfloor\lambda\rfloor)^n$ since each edge of the cube $\lfloor\lambda\rfloor\mathcal{B}_\infty$ contains $(1+2\lfloor\lambda\rfloor)$ integer points. Bounds for $\#(\lambda\mathcal{B}_2\cap\mathbb{Z}^n)$ follow by noting $\mathcal{B}_1\subseteq\mathcal{B}_2\subseteq\mathcal{B}_\infty$.
\end{proof}

\begin{remark}
The relevant outcome for our purposes is that $\#(\lambda\mathcal{B}_1\cap\mathbb{Z}^n)$ is polynomial in $n$ when $\lambda$ is fixed.  As is well known, the above result says that $\#(\lambda\mathcal{B}_\infty\cap\mathbb{Z}^n)$ is exponential in $n$.  Our bounds for the $\ell_2$-ball are ambiguous, though we conjecture that $\#(\lambda\mathcal{B}_2\cap\mathbb{Z}^n)$ is exponential in $n$.
\end{remark}

\begin{remark}
A simplified set of bounds for $\lambda \geq 1$ implied by the above theorem are
\begin{equation}
\textstyle 2n\leq\#(\lambda\mathcal{B}_1\cap\mathbb{Z}^n) \leq n^{4\lambda^2}.
\end{equation}
\end{remark}


\subsection{Algorithm for Solving L1-Constrained Integer Programs}

\begin{algorithm}[t]
\caption{Solve L1-constrained integer programs}
\label{alg:enum1}
\begin{algorithmic}
\State $f^* := +\infty$
\ForAll{$u \in M_n^{\lfloor\lambda\rfloor +1}$}
\State $v := \varphi(u)$
\ForAll{$i\in[n]$}
\State $b_i := \mathbf{1}(v_i \neq 0)$
\EndFor
\ForAll{$s \in \mathrm{SignPerm}(b)$}
\State $x := v\circ s$
\If{$f(x) < f^*\ \mathrm{and}\ g(x) \leq 0$}
\State $x^* = x$
\State $f^* = f(x)$
\EndIf
\EndFor
\EndFor
\end{algorithmic}
\end{algorithm}

Let $M^N_k$ be the set of vectors with values corresponding to the $k$-multisets of set $[N]$ (bags with $k$ values chosen with replacement from the set $[N]$), such that for all $u \in M^N_k$ we have $u_i \leq u_{i+1} \ \forall i \in[k-1]$.  Standard algorithms can generate each single combination in $O(1)$ oracle time, $O(\log N)$ arithmetic time, and $O(k)$ memory \cite{takaoka1999}.  Let $\mathrm{SignPerm}(b)$ be the set of all sign ($+/-$) permutations of the nonzero entries of the binary vector $b$.  Standard algorithms can generate each single permutation in $O(1)$ oracle time, $O(\log k)$ arithmetic time, and $O(\sum_i b_i)$ memory \cite{knuth2014}.

Next, we define the function $\varphi : M_k^N\rightarrow\mathbb{N}^{k}$ such that for any $u \in M_k^N$ the function is given by $(\varphi(u))_1 = u_1 -1$ and $(\varphi(u))_i = u_i - u_{i-1}$ for $i\in\{2,\ldots,k\}$.  This function maps each $u\in M^N_k$ to a corresponding point in $(N-1)\mathcal{B}_1 \cap \mathbb{N}^{k}$.  The below result shows this function is a bijection and is computable in linear time:

\begin{lemma} \label{lemma:phi_lemma}
The function $\varphi$ provides a bijection between $M_k^N$ and $(N-1)\mathcal{B}_1 \cap \mathbb{N}^{k}$, and it is computable in $O(k)$ oracle steps and $O(k\log N)$ arithmetic steps.
\end{lemma}

\begin{proof}
Consider a vector $u \in M_k^N$.  By construction each component of $\varphi(u)$ is a nonnegative integer, and we have $\sum_{i=1}^{k}(\varphi(u))_i = u_1 - 1 + \sum_{i=2}^k(u_i - u_{i-1}) = u_k-1\leq N-1$.  Thus $\varphi(u)\in (N-1)\mathcal{B}_1 \cap \mathbb{N}^{k}$.  To show $\varphi$ is bijective, observe that it can be written as $\varphi(u) = Mu - e_1$ where $M$ is a lower bidiagonal (square) matrix with 1 on the main diagonal and $-1$ on the diagonal below, and $e_1$ is a vector whose first entry is 1 and the remaining entries are 0. The matrix $M$ can be seen to have full rank, and so $\varphi$ must be bijective.  The computational complexity follows because we perform one subtraction operation per entry of the result vector, which has dimension $k$.
\end{proof}

\begin{theorem} \label{thm:bigo}
Algorithm \ref{alg:enum1} solves (\ref{eqn:gip}) in $O(n^{((2+\delta)\lfloor\lambda\rfloor)^2/2+1})$ oracle steps, for any $\delta \in (0,1)$ and where computing $f,g$ comprises a single oracle step.
\end{theorem}
\begin{proof}
Recall that $\#(\lambda\mathcal{B}_1\cap\mathbb{Z}^n) \leq n^{((2+\delta)\lfloor\lambda\rfloor)^2/2}$ by Corollary \ref{cor:intbound}, the sets $M_n^{\lfloor\lambda\rfloor +1}$, $\mathrm{SignPerm}(h)$ are of finite cardinality and such that each single combination can be computed in $O(1)$ oracle time \cite{knuth2014,takaoka1999}, and computing $v$ requires $O(n)$ oracle time by Lemma \ref{lemma:phi_lemma}.  Hence it suffices to show that Algorithm \ref{alg:enum1} will enumerate over every point $x \in \lambda \mathcal{B}_1 \cap \mathbb{Z}^n$ exactly once. Let $\mathcal{A}$ be the set of points that are enumerated by Algorithm \ref{alg:enum1}. Suppose $x \in \mathcal{A}$ then by the second inner loop this must mean that $|x| \in \mathcal{A}$. Since $|x| = \varphi(u)$ for some $u \in M_n^{\lfloor\lambda\rfloor+1}$ this means by Lemma \ref{lemma:phi_lemma} that $\sum_{i=1}^{n}|x|_i \leq \lambda $. Hence $x \in \lambda \mathcal{B}_1 \cap \mathbb{Z}^n$ and $\mathcal{A} \subseteq \lambda \mathcal{B}_1 \cap \mathbb{Z}^n$. Now suppose $x \in \lambda \mathcal{B}_1 \cap \mathbb{Z}^n$, then $|x| \in \lambda\mathcal{B}_1 \cap \mathbb{N}^n$. By Lemma \ref{lemma:phi_lemma}, $\varphi$ is bijective: Therefore, there exists a unique $u \in M_n^{\lfloor \lambda \rfloor +1}$ such that $u = \varphi_0^{-1}(|x|)$. Since Algorithm \ref{alg:enum1} iterates over all $u \in M_n^{\lfloor\lambda\rfloor+1}$, $|x|$ must be generated by some iteration of the second inner loop. Hence $\lambda\mathcal{B}_1 \cap \mathbb{Z}^n \subseteq \mathcal{A}$, and so $\mathcal{A} = \lambda\mathcal{B}_1 \cap \mathbb{Z}^n$ since we have shown both set inclusions. Note that since $\varphi$ is a bijection and each permutation vector in $\mathrm{SignPerm}$ is distinct, this means each point cannot be enumerated more then once. 
\end{proof}

\begin{remark}
A simplified result implied by the above theorem is that the oracle complexity is $O(n^{4\lambda^2+1})$.
\end{remark}

\begin{corollary}
\label{cor:algenum1}
If $f,g$ are computable in polynomial time, then Algorithm \ref{alg:enum1} solves (\ref{eqn:gip}) in polynomial time for fixed $\lambda$.  
When $s$ is the most bits needed to represent the objective or a constraint, and $\delta \in (0,1)$; then some specific cases are:
\begin{enumerate}[leftmargin=0.5cm]
\item Integer Linear Program (ILP): If $f = c'x$ and $g = Ax-b$, where $A\in\mathbb{Q}^{m\times n}$, $b\in\mathbb{Q}^m$, and $c\in\mathbb{Q}^n$; then Algorithm \ref{alg:enum1} solves (\ref{eqn:gip}) in $O(s^2 mn^{((2+\delta)\lfloor\lambda\rfloor)^2/2+1})$ time.
\item (Non-convex) Integer Quadratic Program (IQP): If $f = x'Q'x + c'x$ and $g_i = x'A_ix+b_i'x+c_i$ for $i \in[m]$, where $A_i,Q\in\mathbb{Q}^{n\times n}$, $b_i,c\in\mathbb{Q}^n$, and $c_i\in\mathbb{Q}$; then Algorithm \ref{alg:enum1} solves (\ref{eqn:gip}) in $O(s^2 mn^{((2+\delta)\lfloor\lambda\rfloor)^2/2 + 2})$ time.
\item (Non-convex) Integer Quadratically-Constrained Quadratic Program (IQCQP): If $f = x'Qx+c'x$ and $g_i = x'A_ix+b_i'x+c_i$ for $i \in[m]$, where $A_i,Q\in\mathbb{Q}^{n\times n}$, $b_i,c\in\mathbb{Q}^n$, and $c_i\in\mathbb{Q}$; then Algorithm \ref{alg:enum1} solves (\ref{eqn:gip}) in $O(s^2mn^{((2+\delta)\lfloor\lambda\rfloor)^2/2 + 2})$ time.
\end{enumerate}
\end{corollary}
\begin{proof}
Using the result of Theorem \ref{thm:bigo} it suffices to compute the time complexity of the oracle for $f,g$ to determine the overall time complexity of each instance, which will result in the form $O((P(n,m,s)+ns)\cdot n^{((2+\delta)\lfloor\lambda\rfloor)^2/2})$, where $P$ is a polynomial of $n,m,s$ and $O(ns)$ is the arithmetic complexity for generating each iteration by Lemma \ref{lemma:phi_lemma} and \cite{knuth2014,takaoka1999}. For ILP, since the dot product of two vectors in $\mathbb{Q}^n$ can be computed in $O(s^2 n)$ time and $g$ is comprised of $m$ constraints this means that the complexity $f,g$ is of order $O(s^2 mn)$ and hence the resulting complexity bound. For IQP, since resolving quadratic forms of matrices in $\mathbb{Q}^{n\times n}$ requires time complexity of $O(s^2 n^2)$ we obtain that the complexity of evaluating $f,g$ is of order $O(s^2 mn^2)$. Likewise, for IQCQP the time complexity of resolving the quadratic constraints dominates the complexity of evaluating the objective function, hence the time complexity of resolving this oracle is $O(s^2 mn^2)$.
\end{proof}

\begin{remark}
A simplified result from the above corollary is that the arithmetic complexity is: $O(s^2 mn^{4\lambda^2+1})$ for ILP, and $O(s^2 mn^{4\lambda^2+2})$ for IQP and IQCQP.
\end{remark}

\subsection{Algorithm for Solving Weighted L1-Constrained Integer Programs}

\label{sect:wel1}

The algorithm and analysis in the previous section easily generalize to the case of an integer program with a weighted L1-constraint:
\begin{equation}
\label{eqn:wgip}
\textstyle\min\{f(x)\ |\ g(x) \leq 0, \sum_iw_i|x_i| \leq \lambda, x\in\mathbb{Z}^n\},
\end{equation}
with $w_i > 0$.  If we define the effective dimension $\overline{n} = \sum_i\mathbf{1}(w_i\leq\lambda)$, the effective radius $\mu = \lambda/(\min_i w_i)$, and the effective decision variable $y\in\mathbb{Z}^{\overline{n}}$ with a bijection between the components of $y$ and the non-zero components of $x$ (i.e., $x_i$ such that $w_i\leq\lambda$); then we can rewrite this problem as
\begin{multline}
\label{eqn:rewgip}
\textstyle\min\{f(My)\ |\ g(My) \leq 0, \sum_iw_i|(My)_i|\leq\lambda, \\
\|y\|_1 \leq \mu, y\in\mathbb{Z}^{\overline{n}}\},
\end{multline}
where $M$ gives the bijection (i.e., $x = My$).  We can then solve this problem by applying Algorithm \ref{alg:enum1}, and similar results to Theorem \ref{thm:bigo} and Corollary \ref{cor:algenum1} can be shown but with $\overline{n}$ and $\mu$ taking the place of $n$ and $\lambda$.

\subsection{Bounding the Running Time of a Given Integer Program}

\label{sect:bound}

Here, we provide a polyomial-time algorithm to generate upper bounds on the number of oracle operations required to solve an integer program
\begin{equation}
\label{eqn:ggip}
\min\{f(x)\ |\ g(x) \leq 0, x\in\mathbb{Z}^n\}
\end{equation}
with a bounded feasible region and convex constraints $g(x)$ that can be optimized in polynomial time.  Let $\mathcal{C} = \{x : g(x) \leq 0\}$ be the continuous relaxation for the feasible region of (\ref{eqn:ggip}).  Our algorithm is deterministic and provides bounds that are polynomial or exponential in dimension $n$.  The trivial bound on complexity for when the feasible region lies within $\lambda\mathcal{B}_\infty$ is $(1+2\lfloor\lambda\rfloor)^n$ (see Corollary \ref{cor:intbound}), and our algorithm can potentially provide polynomial in $n$ bounds for specific problem instances.  

\begin{algorithm}[t]
\caption{Bound running time of integer programs with convex constraints}
\label{alg:bound2}
\begin{algorithmic}
\ForAll{$i\in[n]$}
\State $l_i := -\min\big\{\min_x\big\{x_i\ \big|\ g(x)\leq0\big\}, 0\big\}$
\State $u_i := \max\big\{\max_x\big\{x_i\ \big|\ g(x)\leq0\big\}, 0\big\}$
\EndFor
\State $\rho := \big\lfloor\textstyle\max_x\big\{\sum_i (s_i+t_i)\ \big|\ g(s-t)\leq0, 0 \leq s \leq u, 0\leq t \leq l\big\}\big\rfloor$
\State $\mathrm{bnd} := n^{((2+\delta)\rho)^2/2+1}$\ \ for any $\delta\in(0,1)$
\end{algorithmic}
\end{algorithm}

Algorithm \ref{alg:bound2} describes our procedure.  The intuition is that the algorithm finds a radius $\rho$ such that $\rho\mathcal{B}_1$ covers $\mathcal{C}$, and then bounds the oracle complexity using Theorem \ref{thm:bigo}.  When the feasible set $\mathcal{C}$ is nonnegative (i.e., $\mathcal{C} \subset \mathbb{R}^n_+$), the minimum radius is readily computed by solving $\max_x\big\{\sum_i x_i\ \big|\ g(x_i)\leq0\big\}$.  When $\mathcal{C}$ is general (i.e., $\mathcal{C}\subset\mathbb{R}^n$), directly finding the minimum radius is difficult in this case.  So the algorithm takes a different approach: It performs a change of variables $x = u - v$ where $u,v\geq 0$, computes bounds on $u,v$, and converts (\ref{eqn:ggip}) into another instance of (\ref{eqn:ggip}) with a nonnegative feasible region.  The correctness and polynomial-time complexity of Algorithm \ref{alg:bound2} is given by the following result: 

\begin{theorem}
Suppose $\mathcal{C} \subset\mathbb{R}^n$ is bounded, and that $g(x)$ can be optimized in polynomial time.  Then Algorithm \ref{alg:bound2}  runs in polynomial time, and the value $\mathrm{bnd}$ computed by the algorithm is an upper bound on the oracle complexity for solving (\ref{eqn:ggip}).
\end{theorem}

\begin{proof}
This algorithm runs in polynomial time because it consists of solving $2n+1$ convex optimization problems with at most $2n$ variables, each of which can be solved in polynomial time by assumption on $g$ and the linearity of the objective functions.  Next we prove that the value $\mathrm{bnd}$ provides an upper bound on the oracle complexity.  If $x = s-t$ with $s,t\geq 0$; then $\|x\|_1 = \|s-t\|_1\leq\|s\|_1+\|t\|_1=\sum_i (s_i+t_i)$, where the last equality holds since $s,t\geq 0$.  And $l,u$ are constructed so that for all $x \in \mathcal{C}$: if $s = \max\{x,0\}$ and $t = \max\{-x,0\}$, then we have $x = s - t$ with $0 \leq s \leq u$ and $0 \leq t \leq l$.  Hence we have $\max_x\big\{\|x\|_1\ \big|\ g(x)\leq0\big\} \leq \max_x\big\{\sum_i (s_i+t_i)\ \big|\ g(s-t)\leq0, 0 \leq s \leq u, 0\leq t \leq l\big\}$.  Thus $\mathcal{C}\cap\mathbb{Z}^n \subseteq\rho\mathcal{B}_1\cap\mathbb{Z}^n$ for $\rho$ as defined in the algorithm.  This means we can solve (\ref{eqn:ggip}) by using Algorithm \ref{alg:enum1} to solve (\ref{eqn:gip}) with $\rho \equiv \lambda$, and so the oracle complexity for solving (\ref{eqn:ggip}) is bounded by the rate given in Theorem \ref{thm:bigo}.
\end{proof}

\begin{remark}
When $\mathcal{C}$ lies in a single orthant, the value $\rho$ is tight in the sense that $\max_x\big\{\|x\|_1\ \big|\ g(x)\leq0\big\} = \max_x\big\{\sum_i (s_i+t_i)\ \big|\ g(s-t)\leq0, 0 \leq s \leq u, 0\leq t \leq l\big\}$.
\end{remark}

\begin{remark}
Algorithm \ref{alg:bound2} can be modified to return simplified bounds by changing the last statement of the algorithm to ``$\mathrm{bnd} := n^{4\rho^2+1}$''.
\end{remark}

The next proposition shows that our algorithm returns non-trivial bounds.  More specifically, it returns polynomial bounds for some instances of (\ref{eqn:ggip}) and exponential bounds for other instances of (\ref{eqn:ggip}).

\begin{proposition}
Suppose $\lambda\in\mathbb{N}$.  If $\mathcal{C} \subseteq\lambda\mathcal{B}_1\cap\mathbb{R}^n_+$, then Algorithm \ref{alg:bound2} returns a value $\mathrm{bnd} = O(n^{4\lambda^2+1})$ that is polynomial in $n$.  If $\mathcal{C} \supseteq\lambda\mathcal{B}_\infty\cap\mathbb{R}^n_+$, then Algorithm \ref{alg:bound2} returns a value $\mathrm{bnd} = \Omega(n^{2(n\lambda)^2+1}) = \Omega((1+2\lambda)^n)$ that is exponential in $n$.  
\end{proposition}

\begin{proof}
In both cases, $\mathcal{C}$ is in the non-negative orthant and so $\max_x\big\{\|x\|_1\ \big|\ g(x)\leq0\big\} = \max_x\big\{\sum_i (s_i+t_i)\ \big|\ g(s-t)\leq0, 0 \leq s \leq u, 0\leq t \leq l\big\}$.  The first case has $\mathcal{C} \subseteq\lambda\mathcal{B}_1\cap\mathbb{R}^n_+$, and so $\max_x\big\{\|x\|_1\ \big|\ g(x)\leq0\big\} \leq \max_x\big\{\|x\|_1\ \big|\ x\in\lambda\mathcal{B}_1\cap\mathbb{R}^n_+\big\} = \lambda$.  Thus the algorithm will compute $\rho \leq \lambda$, which means it will return $\mathrm{bnd} \leq n^{((2+\delta)\lambda)^2/2+1}$ for any $\delta \in (0,1)$. The second case has $\mathcal{C} \supseteq\lambda\mathcal{B}_\infty\cap\mathbb{R}^n_+$, and so $\max_x\big\{\|x\|_1\ \big|\ g(x)\leq0\big\} \geq \max_x\big\{\|x\|_1\ \big|\ x\in\lambda\mathcal{B}_\infty\cap\mathbb{R}^n_+\big\} = n\lambda$.  Thus the algorithm will compute $\rho \geq n\lambda$, which means it will return $\mathrm{bnd} \geq n^{((2+\delta)n\lambda)^2/2+1}$ for any $\delta \in (0,1)$. 
\end{proof}

\begin{remark}
We have used simplified bounds, without $\delta \in (0,1)$, when stating the above proposition.
\end{remark}

\begin{remark}
Algorithm \ref{alg:bound2} returns a hyper-exponential bound $\mathrm{bnd} = \Omega(n^{2(n\lambda)^2+1})$ when $\mathcal{C} \supseteq\lambda\mathcal{B}_\infty\cap\mathbb{R}^n_+$, and we include the weaker (exponential in $n$) bound $\Omega((1+2\lambda)^n)$ to emphasize that our algorithm provides a bound that is consistent with Corollary \ref{cor:intbound}.
\end{remark}

\section{PTAS for Optimizing Lipschitz Problems over the Scaled L1-Ball}
\label{section:ptas}

\begin{algorithm}[t]
\caption{PTAS for continuous optimization of Lipschitz problems}
\label{alg:enum2}
\begin{algorithmic}
\State $f^* := +\infty$
\ForAll{$u \in M_n^{\lfloor\lambda\kappa/\epsilon\rfloor +1}$}
\State $v := \varphi(u)$
\ForAll{$i\in[n]$}
\State $b_i := \mathbf{1}(v_i \neq 0)$
\EndFor
\ForAll{$s \in \mathrm{SignPerm}(b)$}
\State $x := \epsilon/\kappa\cdot(v\circ s)$
\If{$f(x) < f^*\ \mathrm{and}\ g(x) \leq \epsilon$}
\State $x^* = x$
\State $f^* = f(x)$
\EndIf
\EndFor
\EndFor
\end{algorithmic}
\end{algorithm}

In this section, we modify Algorithm \ref{alg:enum1} in order to develop an additive PTAS for \eqref{eqn:formulation}.  Let $x^*$ be any minimizer of (\ref{eqn:formulation}).  Then we define an additive PTAS for (\ref{eqn:formulation}) to be an algorithm that for fixed $\epsilon > 0$ requires a polynomial in $n$ number of oracle operations to compute a solution $\hat{x}$ such that $f(\hat{x}) - f(x^*) \leq \epsilon$, $g(\hat{x})\leq\epsilon$, $\|\hat{x}\|_1\leq\lambda$, and $\hat{x}\in\mathbb{R}^n$.  Recall $\kappa \in\mathbb{R}_+$ is the Lipschitz constant in the sense: $|f(x)-f(y)|\leq\kappa\|x-y\|_\infty$ and $|g_i(x)-g_i(y)|\leq\kappa\|x-y\|_\infty$ for all $i$.  Our Algorithm \ref{alg:enum2} finds such an $\hat{x}$ by enumerating over a set of points that forms an $(\epsilon/\kappa)\mathcal{B}_\infty$ cover of the feasible region in polynomial time, and our final result formally proves this:


\begin{theorem}
If $\epsilon > 0$, and an optimal $x^*$ exists for \eqref{eqn:formulation}; then Algorithm \ref{alg:enum2} is an additive PTAS for \eqref{eqn:formulation} with oracle time complexity $O(n^{((2+\delta)\lfloor\lambda \kappa/\epsilon\rfloor)^2/2+1})$ for any $\delta \in (0,1)$.
\end{theorem}
\begin{proof}
Algorithm \ref{alg:enum2} is a modified version of Algorithm \ref{alg:enum1}, and so the same argument from the proof of Theorem \ref{thm:bigo} implies $v\circ s$ enumerates over all points in $(\lambda\kappa/\epsilon) \mathcal{B}_1\cap \mathbb{Z}^n$.  This means the $x$ in Algorithm \ref{alg:enum2} are such that $\|x\|_1 = (\epsilon/\kappa)\cdot\|v \circ s\|_1 = (\epsilon/\kappa)\cdot\|v\|_1 \leq (\epsilon/\kappa)\cdot\lfloor\lambda\kappa/\epsilon\rfloor \leq \lambda$.  Suppose these $x$ are the centers of copies of $(\epsilon/\kappa)\mathcal{B}_\infty$ that form a covering of $\lambda\mathcal{B}_1$, then there exists a covering ball containing $x^*$.  Let $\tilde{x}$ be the center of this ball, and observe that $f(\tilde{x}) - f(x^*) \leq \kappa\|\tilde{x}-x^*\|_\infty \leq \epsilon$ and similarly $g_i(\tilde{x}) \leq g_i(x^*) + |g_i(\tilde{x}) - g_i(x^*)| \leq \kappa\|\tilde{x}-x^*\|_\infty \leq \epsilon$ for all $i$.  This would mean that a solution $\hat{x}$ with the desired properties is returned by Algorithm \ref{alg:enum2}.  Furthermore, the same argument for Theorem \ref{thm:bigo} yields that the oracle run time complexity for Algorithm \ref{alg:enum2} is $O(n^{((2+\delta)\lfloor\lambda \kappa/\epsilon\rfloor)^2/2+1})$.

Hence the result follows if we can show that the set of all points $x$ generated by the algorithm are the centers of $(\epsilon/\kappa)\mathcal{B}_\infty$ copies that form a covering of $\lambda\mathcal{B}_1$.  Consider any $y \in \lambda\mathcal{B}_1$, and note $\big\lfloor \kappa |y|/\epsilon\big\rfloor \in (\lambda\kappa/\epsilon) \mathcal{B}_1\cap \mathbb{Z}^n_+$.  Thus Algorithm \ref{alg:enum2} must choose $v$ such that $v = \big\lfloor \kappa |y|/\epsilon\big\rfloor$ in some iteration by the same argument from the proof of Theorem \ref{thm:bigo}.  If we let $s = \text{sign}(y)$, then for this $v,s$ the corresponding value of $x$ chosen by the algorithm is $x = (\epsilon/\kappa)\cdot(v\circ s) = (\epsilon/\kappa)\cdot \big\lfloor \kappa |y|/\epsilon\big\rfloor\circ s$.  Thus $\big\|y - x\big\|_\infty = \big\|y - (\epsilon/\kappa)\cdot \big\lfloor \kappa |y|/\epsilon\big\rfloor\circ s\big\|_\infty = \big\||y|\circ s - (\epsilon/\kappa)\cdot \big\lfloor \kappa |y|/\epsilon\big\rfloor\circ s\big\|_\infty \leq \epsilon/\kappa$.  Restated, this argument shows that for any point $y \in \lambda\mathcal{B}_1$, Algorithm \ref{alg:enum2} generates  an $x$ such that $\|y - x\|_\infty \leq \epsilon/\kappa$.  Thus the set of all points $x$ generated by the algorithm are the centers of $(\epsilon/\kappa)\mathcal{B}_\infty$ copies that form a covering of $\lambda\mathcal{B}_1$.
\end{proof}

\begin{remark}
A simplified result implied by the above theorem for $\epsilon > 0$ is that the oracle complexity is $O(n^{4(\lambda \kappa/\epsilon)^2+1})$
\end{remark}

\begin{remark}
Algorithm \ref{alg:enum2} can be modified as in Sect. \ref{sect:wel1} to solve minimization with weighted L1-constraints.
\end{remark}

\begin{remark}
Our results generalize to the mixed-integer optimization problem given by $\min \big\{f(x,y)\ \big|\ g(x,y) \leq 0, \|x\|_1 \leq \lambda, x\in\mathbb{Z}^n, y\in\mathbb{R}^m\big\}$ when the functions $f,g$ are convex in $y$ for each fixed $x$.  In particular, we can use Algorithm \ref{alg:enum1} to enumerate over all possible $x$, and for each fixed $x$ we solve a convex optimization problem.
\end{remark}

\section{Conclusion}

Using a geometric argument based on stochastic process theory and the theory of convex bodies in high-dimensional spaces, we showed the number of integers within a scaled $\ell_1$-ball is polynomial in dimension when the radius of the ball is fixed.  This result was used to develop an algorithm that solves L1-constrained integer programs and has oracle complexity that is polynomial in dimension when the radius of the L1-constraint is fixed. Our result implies polynomial arithmetic time complexity for integer programming with fixed radius L1-constraints and polynomial-time computable objective function and constraints. Next we used these results to develop an additive PTAS for continuous optimization of problems with Lipschitz objective and constraints intersected with an L1-constraint, and we briefly sketched how these approaches generalize to mixed-integer optimization.

\bibliographystyle{IEEEtran}
\bibliography{IEEEabrv,l1opt}

\begin{thebibliography}{10}
\providecommand{\url}[1]{#1}
\csname url@samestyle\endcsname
\providecommand{\newblock}{\relax}
\providecommand{\bibinfo}[2]{#2}
\providecommand{\BIBentrySTDinterwordspacing}{\spaceskip=0pt\relax}
\providecommand{\BIBentryALTinterwordstretchfactor}{4}
\providecommand{\BIBentryALTinterwordspacing}{\spaceskip=\fontdimen2\font plus
\BIBentryALTinterwordstretchfactor\fontdimen3\font minus
  \fontdimen4\font\relax}
\providecommand{\BIBforeignlanguage}[2]{{%
\expandafter\ifx\csname l@#1\endcsname\relax
\typeout{** WARNING: IEEEtran.bst: No hyphenation pattern has been}%
\typeout{** loaded for the language `#1'. Using the pattern for}%
\typeout{** the default language instead.}%
\else
\language=\csname l@#1\endcsname
\fi
#2}}
\providecommand{\BIBdecl}{\relax}
\BIBdecl

\bibitem{xu2004optimal}
X.~Xu and P.~J. Antsaklis, ``Optimal control of switched systems based on
  parameterization of the switching instants,'' \emph{IEEE transactions on
  automatic control}, vol.~49, no.~1, pp. 2--16, 2004.

\bibitem{sun2005analysis}
Z.~Sun and S.~S. Ge, ``Analysis and synthesis of switched linear control
  systems,'' \emph{Automatica}, vol.~41, no.~2, pp. 181--195, 2005.

\bibitem{aswani2009monotone}
A.~Aswani and C.~Tomlin, ``Monotone piecewise affine systems,'' \emph{IEEE
  Trans. Autom. Control}, vol.~54, no.~8, pp. 1913--1918, 2009.

\bibitem{vasudevan2013consistent}
R.~Vasudevan, H.~Gonzalez, R.~Bajcsy, and S.~S. Sastry, ``Consistent
  approximations for the optimal control of constrained switched systems---part
  1: A conceptual algorithm,'' \emph{SIAM J Control Optim}, vol.~51, no.~6, pp.
  4463--4483, 2013.

\bibitem{vasudevan2013consistent2}
------, ``Consistent approximations for the optimal control of constrained
  switched systems---part 2: An implementable algorithm,'' \emph{SIAM J Control
  Optim}, vol.~51, no.~6, pp. 4484--4503, 2013.

\bibitem{wardi2015switched}
Y.~Wardi, M.~Egerstedt, and M.~Hale, ``Switched-mode systems: gradient-descent
  algorithms with armijo step sizes,'' \emph{Discrete Event Dynamic Systems},
  vol.~25, no.~4, pp. 571--599, 2015.

\bibitem{ribbing2007lasso}
J.~Ribbing, J.~Nyberg, O.~Caster, and E.~N. Jonsson, ``The lasso--a novel
  method for predictive covariate model building in nonlinear mixed effects
  models,'' \emph{Journal of pharmacokinetics and pharmacodynamics}, vol.~34,
  no.~4, pp. 485--517, 2007.

\bibitem{buhlmann2011non}
P.~B{\"u}hlmann and S.~van~de Geer, ``Non-convex loss functions and
  ℓ1-regularization,'' in \emph{Statistics for High-Dimensional Data}.\hskip
  1em plus 0.5em minus 0.4em\relax Springer, 2011, pp. 293--338.

\bibitem{glorot2011deep}
X.~Glorot, A.~Bordes, and Y.~Bengio, ``Deep sparse rectifier neural networks.''
  in \emph{Aistats}, vol.~15, no. 106, 2011, p. 275.

\bibitem{srivastava2014dropout}
N.~Srivastava, G.~E. Hinton, A.~Krizhevsky, I.~Sutskever, and R.~Salakhutdinov,
  ``Dropout: a simple way to prevent neural networks from overfitting.''
  \emph{Journal of Machine Learning Research}, vol.~15, no.~1, pp. 1929--1958,
  2014.

\bibitem{raghavan1987randomized}
P.~Raghavan and C.~D. Tompson, ``Randomized rounding: a technique for provably
  good algorithms and algorithmic proofs,'' \emph{Combinatorica}, vol.~7,
  no.~4, pp. 365--374, 1987.

\bibitem{raghavan1988probabilistic}
P.~Raghavan, ``Probabilistic construction of deterministic algorithms:
  approximating packing integer programs,'' \emph{Journal of Computer and
  System Sciences}, vol.~37, no.~2, pp. 130--143, 1988.

\bibitem{srinivasan1999improved}
A.~Srinivasan, ``Improved approximation guarantees for packing and covering
  integer programs,'' \emph{SIAM Journal on Computing}, vol.~29, no.~2, pp.
  648--670, 1999.

\bibitem{downey1999}
R.~G. Downey and M.~R. Fellows, \emph{Parameterized complexity}.\hskip 1em plus
  0.5em minus 0.4em\relax Springer Science \& Business Media, 1999.

\bibitem{niedermeier2006}
R.~Niedermeier, \emph{Invitation to Fixed-Parameter Algorithms}.\hskip 1em plus
  0.5em minus 0.4em\relax OUP Oxford, 2006.

\bibitem{borosh1976}
I.~Borosh and L.~B. Treybig, ``Bounds on positive integral solutions of linear
  diophantine equations,'' \emph{Proceedings of the American Mathematical
  Society}, vol.~55, no.~2, pp. 299--304, 1976.

\bibitem{von1978}
J.~von~zur Gathen and M.~Sieveking, ``A bound on solutions of linear integer
  equalities and inequalities,'' \emph{Proceedings of the American Mathematical
  Society}, vol.~72, no.~1, pp. 155--158, 1978.

\bibitem{kannan1978}
R.~Kannan and C.~L. Monma, ``On the computational complexity of integer
  programming problems,'' in \emph{Optimization and Operations Research}.\hskip
  1em plus 0.5em minus 0.4em\relax Springer, 1978, pp. 161--172.

\bibitem{papadimitriou1981}
C.~H. Papadimitriou, ``On the complexity of integer programming,''
  \emph{Journal of the ACM (JACM)}, vol.~28, no.~4, pp. 765--768, 1981.

\bibitem{pia2016}
A.~D. Pia, S.~S. Dey, and M.~Molinaro, ``Mixed-integer quadratic programming is
  in {NP},'' \emph{Mathematical Programming}, pp. 1--16, 2016.

\bibitem{clarkson1995}
K.~L. Clarkson, ``Las {Vegas} algorithms for linear and integer programming
  when the dimension is small,'' \emph{Journal of the ACM (JACM)}, vol.~42,
  no.~2, pp. 488--499, 1995.

\bibitem{eisenbrand2003}
F.~Eisenbrand, ``Fast integer programming in fixed dimension,'' in
  \emph{European Symposium on Algorithms}.\hskip 1em plus 0.5em minus
  0.4em\relax Springer, 2003, pp. 196--207.

\bibitem{frank1987}
A.~Frank and {\'E}.~Tardos, ``An application of simultaneous diophantine
  approximation in combinatorial optimization,'' \emph{Combinatorica}, vol.~7,
  no.~1, pp. 49--65, 1987.

\bibitem{kannan1987}
R.~Kannan, ``Minkowski's convex body theorem and integer programming,''
  \emph{Mathematics of operations research}, vol.~12, no.~3, pp. 415--440,
  1987.

\bibitem{lenstra1983}
H.~W. Lenstra~Jr, ``Integer programming with a fixed number of variables,''
  \emph{Math Oper Res}, vol.~8, no.~4, pp. 538--548, 1983.

\bibitem{de2006}
J.~A. De~Loera, R.~Hemmecke, M.~K{\"o}ppe, and R.~Weismantel, ``Integer
  polynomial optimization in fixed dimension,'' \emph{Mathematics of Operations
  Research}, vol.~31, no.~1, pp. 147--153, 2006.

\bibitem{hildebrand2016}
R.~Hildebrand, R.~Weismantel, and K.~Zemmer, ``An {FPTAS} for minimizing
  indefinite quadratic forms over integers in polyhedra,'' in \emph{Proceedings
  of the Twenty-Seventh Annual ACM-SIAM Symposium on Discrete
  Algorithms}.\hskip 1em plus 0.5em minus 0.4em\relax SIAM, 2016, pp.
  1715--1723.

\bibitem{del2016}
A.~Del~Pia, R.~Hildebrand, R.~Weismantel, and K.~Zemmer, ``Minimizing cubic and
  homogeneous polynomials over integers in the plane,'' \emph{Math Oper Res},
  vol.~41, no.~2, pp. 511--530, 2016.

\bibitem{deklerk2008}
E.~De~Klerk, D.~Den~Hertog, and G.~Elabwabi, ``On the complexity of
  optimization over the standard simplex,'' \emph{European journal of
  operational research}, vol. 191, no.~3, pp. 773--785, 2008.

\bibitem{deklerk2006}
E.~De~Klerk, M.~Laurent, and P.~A. Parrilo, ``A {PTAS} for the minimization of
  polynomials of fixed degree over the simplex,'' \emph{Theoretical Computer
  Science}, vol. 361, no.~2, pp. 210--225, 2006.

\bibitem{deklerk2015}
E.~de~Klerk, M.~Laurent, and Z.~Sun, ``An alternative proof of a {PTAS} for
  fixed-degree polynomial optimization over the simplex,'' \emph{Mathematical
  Programming}, vol. 151, no.~2, pp. 433--457, 2015.

\bibitem{boucheron2013}
S.~Boucheron, G.~Lugosi, and P.~Massart, \emph{Concentration Inequalities: A
  Nonasymptotic Theory of Independence}.\hskip 1em plus 0.5em minus 0.4em\relax
  OUP Oxford, 2013.

\bibitem{ledoux1991}
M.~Ledoux and M.~Talagrand, \emph{Probability in Banach Spaces: Isoperimetry
  and Processes}, ser. A Series of Modern Surveys in Mathematics Series.\hskip
  1em plus 0.5em minus 0.4em\relax Springer, 1991.

\bibitem{sudakov1971}
V.~Sudakov, ``Gaussian random processes and solid angle measures in {Hilbert}
  space,'' \emph{Doklady Akademii Nauk {SSSR}}, vol. 197, no.~1, p.~43, 1971.

\bibitem{vershynin2009}
R.~Vershynin, ``Lectures in geometric functional analysis,'' University of
  Michigan, Tech. Rep., 2009.

\bibitem{takaoka1999}
T.~Takaoka, ``O(1) time algorithms for combinatorial generation by tree
  traversal,'' \emph{Comput. J.}, vol.~42, no.~5, pp. 400--408, 1999.

\bibitem{knuth2014}
D.~Knuth, \emph{The Art of Computer Programming, Volume 4A: Combinatorial
  Algorithms}.\hskip 1em plus 0.5em minus 0.4em\relax Pearson Education, 2014,
  no. pt. 1.

\end{thebibliography}

\end{document}